\documentclass[graybox]{svmult}

\usepackage{type1cm}        
\usepackage{makeidx}         
\usepackage{graphicx}        
\usepackage{multicol}        
\usepackage[bottom]{footmisc}

\usepackage{newtxtext}       %
\usepackage[varvw]{newtxmath}       

\usepackage{mathrsfs}
\newcommand{\N}{\mathbb{N}}
\newcommand{\Z}{\mathbb{Z}}
\newcommand{\R}{\mathbb{R}}

\makeindex             

\theoremstyle{definition}
\newtheorem{exampl}{Example}

\begin{document}

\title*{A note on composition operators between weighted spaces of smooth functions}
\author{Andreas Debrouwere and Lenny Neyt}
\institute{Andreas Debrouwere \at Department of Mathematics and Data Science, Vrije Universiteit Brussel, Pleinlaan 2, 1050 \\ Brussels, Belgium, \email{andreas.debrouwere@vub.be}
\and Lenny Neyt \at Department of Mathematics: Analysis, Logic and Discrete Mathematics, Ghent University, \\ Krijgslaan 281, 9000 Ghent, Belgium, \email{lenny.neyt@UGent.be}}
%
%


\maketitle

\abstract{For certain weighted locally convex spaces $X$ and $Y$ of one real variable smooth functions, we  characterize the smooth functions $\varphi: \R \to \R$ for which the composition operator $C_\varphi: X \to Y, \, f \mapsto f \circ \varphi$ is well-defined and continuous. 
This  problem has been recently considered for $X = Y$ being the space $\mathscr{S}$ of rapidly decreasing smooth functions \cite{GJ} and the space $\mathscr{O}_M$ of slowly increasing smooth functions \cite{AJM}. In particular, we recover both these results as well as obtain a  characterization  for $X =Y$ being the space $\mathscr{O}_C$ of very slowly increasing smooth functions. 
}

\section{Introduction}
One of the most fundamental questions in the study of  composition operators is to characterize when such an operator is well-defined and continuous in terms of its symbol. The goal of this article is to consider this question for weighted locally convex spaces of one real variable smooth functions. 

Let $\varphi: \R \to \R$ be smooth. In \cite{GJ} Galbis and Jord\'a showed that the composition operator $C_\varphi: \mathscr{S} \to \mathscr{S}, \, f \mapsto f \circ \varphi$, with $\mathscr{S}$ the space of rapidly decreasing smooth functions \cite{Schwartz},  is well-defined (continuous) if and only if
$$
\exists N \in \Z_+~:~ \sup_{x \in \R} \frac{1+ |x|}{(1+|\varphi(x)|)^N} < \infty
$$
and
$$
 \forall p\in \Z_+~\exists N \in \N~:~ \sup_{x \in \R} \frac{|\varphi^{(p)}(x)|}{(1+|\varphi(x)|)^{N}} < \infty.
$$
Albanese et.\ al  \cite{AJM} proved that the composition operator $C_\varphi: \mathscr{O}_M \to \mathscr{O}_M$, with $\mathscr{O}_M$ the space of slowly increasing smooth functions \cite{Schwartz},  is well-defined (continuous) if and only if
$\varphi \in \mathscr{O}_M$. In \cite[Remark 2.6]{AJM} they also pointed out that the corresponding result for the space  $\mathscr{O}_C$ of very slowly increasing smooth functions \cite{Schwartz} is false, namely, they showed that $\sin (x^2) \notin  \mathscr{O}_C$, while, obviously, $\sin x, x^2 \in \mathscr{O}_C$. 

Inspired by these results, we study in this article  the following general question: Given two weighted locally convex spaces $X$ and $Y$ of smooth functions,  when is the composition operator $C_\varphi: X \to Y$  well-defined (continuous)? We shall consider this problem for $X$ and $Y$ both being Fr\'echet spaces, $(LF)$-spaces, or  $(PLB)$-spaces.  

We now state a particular instance of our main result that covers many well-known spaces. We need some preparation. 
Given a positive continuous function $v$ on $\R$, we write $\mathscr{B}^n_v$, $n \in \N$, for the Banach space consisting of all $f \in C^n(\R)$ such that
$$
\|f\|_{v,n} = \max_{p \leq n} \sup_{x \in \R} \frac{|f^{(p)}(x)|}{v(x)} < \infty.
$$
For $v \geq 1$ we consider the following three weighted spaces of smooth functions
\begin{align*}
\mathscr{K}_{v} &= \varprojlim_{N \in \N} \mathscr{B}^N_{1/v^N} , \\
\mathscr{O}_{C,v} &= \varinjlim_{N \in \N} \varprojlim_{n \in \N} \mathscr{B}^n_{v^N} ,\\
\mathscr{O}_{M,v} &=  \varprojlim_{n \in \N} \varinjlim_{N \in \N}  \mathscr{B}^n_{v^N} .
\end{align*}
Theorem \ref{main} below implies the following result:
 
 \begin{theorem} \label{intro}  Let $v,w: \R \to [1,\infty)$ be continuous functions such that 
$$
\sup_{x,t \in \R,|t| \leq 1} \frac{v(x+t)}{v^\lambda(x)} < \infty \qquad \mbox{and} \qquad \sup_{x,t \in \R,|t| \leq 1} \frac{w(x+t)}{w^\mu(x)} < \infty,
$$
 for some $\lambda,\mu > 0$. Let $\varphi: \R \to \R$ be smooth. Then,
\begin{enumerate}
\item[$(I)$]  The following statements are equivalent:
	\begin{enumerate}
	\item[$(i)$] $C_\varphi(\mathscr{K}_{v} ) \subseteq \mathscr{K}_{w}$.
	\item[$(ii)$] $C_\varphi: \mathscr{K}_{v}  \rightarrow \mathscr{K}_{w}$ is continuous.
	\item[$(iii)$] $\varphi$ satisfies the following two properties
		\begin{enumerate}
		\item [$(a)$] $\displaystyle \exists \lambda > 0~:~ \sup_{x \in \R} \frac{w(x)}{v^\lambda(\varphi(x))} < \infty$. 
		\item [$(b)$] $\displaystyle \forall p \in \Z_+~\exists \lambda > 0~:~ \sup_{x \in \R} \frac{|\varphi^{(p)}(x)|}{v^\lambda(\varphi(x))} < \infty$.
		\end{enumerate}
	\end{enumerate}
\item[$(II)$]  The following statements are equivalent:
	\begin{enumerate}
	\item[$(i)$] $C_\varphi(\mathscr{O}_{C,v} ) \subseteq \mathscr{O}_{C,w}$.
	\item[$(ii)$] $C_\varphi: \mathscr{O}_{C,v}  \rightarrow \mathscr{O}_{C,w}$ is continuous.
	\item[$(iii)$] $\varphi$ satisfies the following two properties
		\begin{enumerate}
		\item [$(a)$] $\displaystyle  \exists \mu > 0~:~ \sup_{x \in \R} \frac{v(\varphi(x))}{w^\mu(x)} < \infty$. 
		\item [$(b)$] $\displaystyle \forall p,k \in \Z_+~:~ \sup_{x \in \R} \frac{|\varphi^{(p)}(x)|}{w^{1/k}(x)} < \infty$.
		\end{enumerate}
	\end{enumerate}

\item[$(III)$]  The following statements are equivalent:
	\begin{enumerate}
	\item[$(i)$] $C_\varphi(\mathscr{O}_{M,v} ) \subseteq \mathscr{O}_{M,w}$.
	\item[$(ii)$] $C_\varphi: \mathscr{O}_{M,v}  \rightarrow \mathscr{O}_{M,w}$ is continuous.
	\item[$(iii)$] $\varphi$ satisfies the following two properties
		\begin{enumerate}
		\item [$(a)$] $\displaystyle  \exists \mu > 0~:~ \sup_{x \in \R} \frac{v(\varphi(x))}{w^\mu(x)} < \infty$. 
		\item [$(b)$] $\displaystyle \forall p \in \Z_+~\exists \mu > 0~:~ \sup_{x \in \R} \frac{|\varphi^{(p)}(x)|}{w^\mu(x)} < \infty$.
		\end{enumerate}
	\end{enumerate}
\end{enumerate}
\end{theorem} 
By setting $v(x) = w(x) = 1+|x|$ in Theorem \ref{intro} we recover the above results about $\mathscr{S}$ and $\mathscr{O}_M$ from \cite{GJ, AJM} as well as the following characterization for the space $\mathscr{O}_C$ of very slowly increasing smooth functions: $C_\varphi: \mathscr{O}_C \to \mathscr{O}_C$ is well defined (continuous) if and only if 
$$
\exists N \in \N~:~ \sup_{x \in \R} \frac{|\varphi(x)|}{(1+|x|)^N} < \infty\qquad \mbox{and} \qquad \forall p,k \in \Z_+~:~ \sup_{x \in \R} \frac{|\varphi^{(p)}(x)|}{(1+|x|)^{1/k}} < \infty.
$$
For $v = w = 1$, Theorem \ref{intro} gives the following result for the Fr\'echet space $\mathscr{B}$ of smooth functions that are bounded together will all their derivatives \cite{Schwartz}: $C_\varphi: \mathscr{B} \to \mathscr{B}$ is well defined (continuous) if and only if $\varphi' \in \mathscr{B}$. Another interesting choice is  $v(x) = w(x) = e^{|x|}$, for which Theorem \ref{main} characterizes composition operators on spaces of exponentially decreasing/increasing smooth functions \cite{hasumi, zielezny}. We leave it to the  reader to explicitly formulate this and other examples.

\section{Statement of the main result}

A pointwise non-decreasing sequence $V = (v_{N})_{N \in \N}$ of positive continuous functions on $\R$ is called a \emph{weight system} if $v_0 \geq 1$  and 
$$
\forall N~\exists M \geq N~:~ \sup_{x,t \in \R,|t| \leq 1} \frac{v_N(x+t)}{v_M(x)} < \infty.
$$
We shall also make use of the following condition on a weight system $V = (v_{N})_{N \in \N}$: 
\begin{equation}
 \forall N,M~ \exists K \geq N,M~:~ \sup_{x \in \R} \frac{v_N(x)v_M(x)}{v_K(x)} < \infty.
\label{cond-mult}
\end{equation}
\begin{exampl} \label{example}
Let $v: \R \to [1,\infty)$ be a continuous function satisfying
$$
\sup_{x,t \in \R,|t| \leq 1} \frac{v(x+t)}{v^N(x)} < \infty.
$$
for some $N \in \N$ (cf.\   Theorem \ref{intro}). Then,
$$V_v = (v^N)_{N \in \N}$$ is a weight system satisfying \eqref{cond-mult}.
\end{exampl}

Recall that for a positive continuous function $v$ on $\R$ and $n \in \N$, we write  $\mathscr{B}^n_v$  for the Banach space consisting of all $f \in C^n(\R)$ such that
$$
\|f\|_{v,n} = \max_{p \leq n} \sup_{x \in \R} \frac{|f^{(p)}(x)|}{v(x)} < \infty.
$$
Let $V = (v_{N})_{N \in \N}$ be a weight system. We shall be concerned with the  following weighted spaces of smooth functions
\begin{align*}
\mathscr{K}_{V} &= \varprojlim_{N \in \N} \mathscr{B}^N_{1 / v_N} , \\
\mathscr{O}_{C,V} &= \varinjlim_{N \in \N} \varprojlim_{n \in \N} \mathscr{B}^n_{v_N} , \\
\mathscr{O}_{M,V} &=  \varprojlim_{n \in \N} \varinjlim_{N \in \N}  \mathscr{B}^n_{v_N} .
\end{align*}
Note that $\mathscr{K}_{V}$ is a Fr\'echet space, $\mathscr{O}_{C,V}$ is an $(LF)$-space, and $\mathscr{O}_{M,V}$ is a $(PLB)$-space. Furthermore, we have the following continuous inclusions
$$
\mathscr{D}(\R) \subset \mathscr{K}_{V}  \subset \mathscr{O}_{C,V} \subset \mathscr{O}_{M,V} \subset C^\infty(\R),
$$
where $\mathscr{D}(\R)$ denotes the space of compactly supported smooth functions.  The spaces $\mathscr{K}_V$ were introduced and studied by Gelfand and Shilov \cite{GS}, while we refer to \cite{DV} for more information on the spaces $\mathscr{O}_{C,V}$. For $N,n \in \N$ fixed we will also need the following spaces
$$
\mathscr{B}_{v_N} = \varprojlim_{n \in \N} \mathscr{B}^n_{v_N}, \qquad \mathscr{O}^n_{M,V} =  \varinjlim_{N \in \N}  \mathscr{B}^n_{v_N}. 
$$
The goal of this article is to show the following result.
 
\begin{theorem} \label{main}
Let $V = (v_{N})_{N \in \N}$ and $W = (w_{M})_{M \in \N}$ be two weight systems and let $\varphi: \R \to \R$ be smooth.

\begin{enumerate}
\item[$(I)$]  Suppose that $V$ satisfies \eqref{cond-mult}. The following statements are equivalent:
	\begin{enumerate}
	\item[$(i)$] $C_\varphi(\mathscr{K}_{V} ) \subseteq \mathscr{K}_{W}$.
	\item[$(ii)$] $C_\varphi: \mathscr{K}_{V}  \rightarrow \mathscr{K}_{W}$ is continuous.
	\item[$(iii)$] $\varphi$ satisfies the following two properties
		\begin{enumerate}
		\item [$(a)$] $\displaystyle \forall M~\exists N~:~ \sup_{x \in \R} \frac{w_M(x)}{v_N(\varphi(x))} < \infty$. 
		\item [$(b)$] $\displaystyle \forall p \in \Z_+~\exists N~:~ \sup_{x \in \R} \frac{|\varphi^{(p)}(x)|}{v_N(\varphi(x))} < \infty$.
		\end{enumerate}
	\end{enumerate}

\item[$(II)$]  Suppose that $W$ satisfies \eqref{cond-mult}. The following statements are equivalent:
	\begin{enumerate}
	\item[$(i)$] $C_\varphi(\mathscr{O}_{C,V} ) \subseteq \mathscr{O}_{C,W}$.
	\item[$(ii)$] $C_\varphi: \mathscr{O}_{C,V}  \rightarrow \mathscr{O}_{C,W}$ is continuous.
	\item[$(iii)$] $\displaystyle \forall N~\exists M$ such that $C_\varphi: \mathscr{B}_{v_N}  \rightarrow \mathscr{B}_{w_M}$ is continuous.
	\item[$(iv)$] $\varphi$ satisfies the following two properties
		\begin{enumerate}
		\item [$(a)$] $\displaystyle \forall N~ \exists M~:~ \sup_{x \in \R} \frac{v_N(\varphi(x))}{w_M(x)} < \infty$. 
		\item [$(b)$] $\displaystyle\exists M~\forall p,k \in \Z_+~:~ \sup_{x \in \R} \frac{|\varphi^{(p)}(x)|}{w^{1/k}_M(x)} < \infty$.
		\end{enumerate}
	\end{enumerate}

\item[$(III)$]  Suppose that $W$ satisfies \eqref{cond-mult}. The following statements are equivalent:
	\begin{enumerate}
	\item[$(i)$] $C_\varphi(\mathscr{O}_{M,V} ) \subseteq \mathscr{O}_{M,W}$.
	\item[$(ii)$] $C_\varphi: \mathscr{O}_{M,V}  \rightarrow \mathscr{O}_{M,W}$ is continuous.
	\item[$(iii)$]$C_\varphi: \mathscr{O}^n_{M,V}  \rightarrow \mathscr{O}^n_{M,W}$ is continuous for all $n \in \N$.
	\item[$(iv)$] $\varphi$ satisfies the following two properties
		\begin{enumerate}
		\item [$(a)$] $\displaystyle \forall N~ \exists M~:~ \sup_{x \in \R} \frac{v_N(\varphi(x))}{w_M(x)} < \infty$. 
		\item [$(b)$] $\displaystyle \forall p \in \Z_+~\exists M~:~ \sup_{x \in \R} \frac{|\varphi^{(p)}(x)|}{w_M(x)} < \infty$.
		\end{enumerate}
	\end{enumerate}
\end{enumerate}
\end{theorem}
The proof of Theorem \ref{main} will be given in the next section. The spaces $\mathscr{K}_{v}$, $\mathscr{O}_{C,v}$ and $\mathscr{O}_{M,v}$ from the introduction can be written as 
$$
\mathscr{K}_{v} = \mathscr{K}_{V_v}, \qquad \mathscr{O}_{C,v} = \mathscr{O}_{C,V_v}, \qquad \mathscr{O}_{M,v} = \mathscr{O}_{M,V_v},
$$
where  $V_v = (v^N)_{N \in \N}$ is the weight system from Example \ref{example}. Hence, Theorem \ref{intro}  is a direct consequence of Theorem \ref{main} with $V = V_v$ and $W = V_w$.

\section{Proof of the main result}

Throughout this section we fix a smooth symbol $\varphi: \R \to \R$.  We need two lemmas in preparation for the proof of Theorem \ref{main}. For $n \in \N$ we set
$$
\| f \|_n = \| f \|_{1,n} = \ \max_{p \leq n} \sup_{x \in \R} |f^{(p)}(x)|.
$$
\begin{lemma}\label{lemma-1} Let $v, \widetilde{v},w$ be three positive continuous functions on $\R$ such that
$$
C_0 = \sup_{x,t \in \R,|t| \leq 1} \frac{v(x+t)}{\widetilde{v}(x)} < \infty.
$$
Let $p,n \in \N$ be such that 
\begin{equation}
\| C_\varphi(f) \|_{w,p} \leq C_1 \| f \|_{\widetilde{v},n}, \qquad \forall f \in \mathscr{D}(\R),
\label{norm-ineq}
\end{equation}
for some $C_1 > 0$. Then,
\begin{equation}
\sup_{x \in \R} \frac{v(\varphi(x))}{w(x)} < \infty ,
\label{ineq-0}
\end{equation}
and, if $p \geq 1$, also 
\begin{equation}
\sup_{x \in \R} \frac{v(\varphi(x))|\varphi'(x)|^{p}}{{w}(x)} < \infty ,
\label{ineq-1}
\end{equation}
and
\begin{equation}
\sup_{x \in \R} \frac{v(\varphi(x))|\varphi^{(p)}(x)|}{{w}(x)} < \infty.
\label{ineq-2}
\end{equation}
\end{lemma}
\begin{proof}
Given $f \in \mathscr{D}(\R)$ with $\operatorname{supp} f \subseteq [-1,1]$, we set $f_x = f(\, \cdot \, -\varphi(x))$ for $x \in \R$. Note that
\begin{equation}
 \| f_x \|_{\widetilde{v},n} \leq \frac{C_0\| f\|_n}{v(\varphi(x))}, \qquad x \in \R.
\label{norm-comp}
\end{equation}
We first show \eqref{ineq-0}. Choose $f \in \mathscr{D}(\R)$ with $\operatorname{supp} f \subseteq [-1,1]$ such that $f(0) = 1$. For all $x \in \R$ it holds that
$$
\| C_\varphi(f_x) \|_{w,p} \geq \frac{C_\varphi(f_x)(x)}{w(x)}  =  \frac{1}{w(x)}. 
$$
Hence, by  \eqref{norm-ineq} and \eqref{norm-comp}, we obtain that
$$
\frac{v(\varphi(x))}{{w}(x)} \leq C_0C_1 \| f\|_n, \qquad \forall x \in \R.
$$
Now assume that $p \geq 1$. We  prove \eqref{ineq-1}. Choose $f \in \mathscr{D}(\R)$ with $\operatorname{supp} f \subseteq [-1,1]$ such that $f^{(j)}(0) = 0$ for $j = 1, \ldots, p-1$ and $f^{(p)}(0)=1$. Fa\`a di Bruno's formula implies that for all $x \in \R$
$$
\| C_\varphi(f_x) \|_{w,p} \geq \frac{|C_\varphi(f_x)^{(p)}(x)|}{w(x)}  =  \frac{|\varphi'(x)|^p}{w(x)}. 
$$
Similarly as in the proof of \eqref{ineq-0}, the result now  follows from \eqref{norm-ineq} and \eqref{norm-comp}. Finally, we show \eqref{ineq-2}. Choose $f \in \mathscr{D}(\R)$ with $\operatorname{supp} f \subseteq [-1,1]$ such that $f'(0) = 1$ and $f^{(j)}(0)=0$ for $j = 2, \ldots, p$. Fa\`a di Bruno's formula implies that for all $x \in \R$ 
$$
\| C_\varphi(f_x) \|_{w,p} \geq \frac{|C_\varphi(f_x)^{(p)}(x)|}{w(x)}  =  \frac{|\varphi^{(p)}(x)|}{w(x)}. 
$$
As before, the result is now a consequence of \eqref{norm-ineq} and \eqref{norm-comp}.
\end{proof}
\begin{lemma}\label{lemma-2} 
Let $v$ and $w$ be positive continuous functions on $\R$. Then,
\begin{enumerate}
\item[$(i)$]  If
$$
\sup_{x \in \R} \frac{v(\varphi(x))}{w(x)} < \infty,
$$
then $C_\varphi : \mathscr{B}^0_v \rightarrow \mathscr{B}^0_w$ is well-defined and continuous. 
\item[$(ii)$] Let $n \in \Z_+$. If
$$
\sup_{x \in \R} \frac{v(\varphi(x))}{w(x)} \prod_{p=1}^n |\varphi^{(p)}(x)|^{k_p} < \infty
$$
for all $(k_1, \ldots, k_n) \in \N^n$ with $\sum_{j = 1}^p jk_j \leq p$ for all $p = 1, \ldots, n$, then $C_\varphi : \mathscr{B}^n_{v} \rightarrow \mathscr{B}^n_{w}$ is well-defined and continuous. 
\end{enumerate}
\end{lemma}
\begin{proof}
$(i)$ Obvious. \\
\noindent $(ii)$ This is a direct consequence of  $(i)$ and Fa\`a di Bruno's formula.
\end{proof}
\begin{proof}[of Theorem \ref{main}]
$(I)$ $(i)\Rightarrow (ii)$: Since $C_\varphi: C^\infty(\R) \to C^\infty(\R)$ is continuous, this follows from the closed graph theorem for Fr\'echet spaces. \\
$(ii)\Rightarrow (iii)$:  For all $p,M \in \N$ there are $n,L \in \N$ such that 
$$
\| C_\varphi(f)\|_{p, 1/w_M} \leq C\| f\|_{n,1/v_L}, \qquad \forall f \in \mathscr{K}_V. 
$$
Choose $N \geq L$ such that
$$
 \sup_{x,t \in \R,|t| \leq 1} \frac{v_L(x+t)}{v_N(x)} =  \sup_{x,t \in \R,|t| \leq 1} \frac{1/v_N(x+t)}{1/v_L(x)}   < \infty.
 $$
 Lemma \ref{lemma-1} with $w = 1/w_M$, $v = 1/v_N$ and $\widetilde{v} = 1/v_L$ yields that
 $$
 \sup_{x \in \R} \frac{w_M(x)}{v_N(\varphi(x))} < \infty
$$
and (recall that $w_M \geq 1$)
$$
 \sup_{x \in \R} \frac{|\varphi^{(p)}(x)|}{v_N(\varphi(x))} < \infty.
 $$
$(iii) \Rightarrow (i)$: As $V$ satisfies \eqref{cond-mult}, this follows from Lemma \ref{lemma-2}.
\\ \\
$(II)$ $(i) \Rightarrow (ii)$: Since $C_\varphi: C^\infty(\R) \to C^\infty(\R)$ is continuous, this follows from De Wilde's closed graph theorem. \\
$(ii) \Rightarrow (iii)$: This is a consequence of  Grothendieck's factorization theorem. \\
$(iii) \Rightarrow (iv)$: Fix an arbitrary $N \in \N$. Choose $L \geq N$ such that
$$
\sup_{x,t \in \R,|t| \leq 1} \frac{v_N(x+t)}{v_L(x)} < \infty.
$$
Choose $K \in \N$ such that $C_\varphi: \mathscr{B}_{v_L} \to \mathscr{B}_{w_K}$ is continuous. For all $m \in \Z_+$ there are $n \in \Z_+$ and $C > 0$ such that
$$
\| C_\varphi(f)\|_{m, w_K} \leq C\| f\|_{n,v_L}, \qquad \forall f \in \mathscr{B}_{v_L}.
$$
 Lemma \ref{lemma-1} with $w = w_K$, $v = v_N$ and $\widetilde{v} = v_L$ yields that
 \begin{equation}
\sup_{x \in \R} \frac{v_N(\varphi(x))}{w_K(x)} < \infty
\label{ineq-proof-1}
\end{equation}
and (recall that $v_N \geq 1$)
 \begin{equation}
\sup_{x \in \R} \frac{|\varphi'(x)|}{{w_K^{1/m}}(x)}  < \infty \qquad \mbox{and} \qquad \sup_{x \in \R} \frac{|\varphi^{(m)}(x)|}{w_K(x)} < \infty.
\label{ineq-proof-2}
\end{equation}
Equation \eqref{ineq-proof-1} shows $(a)$. We now prove $(b)$. To this end, we will make use of the following Landau-Kolmogorov type inequality due to Gorny \cite{gorny}:  For all 
$j \leq m \in \Z_+$ there is $C > 0$ such that
\begin{equation}
\| g^{(j)}\| \leq C \| g\|^{1 - j/m} \left(\max \{ \| g \| ,  \| g^{(m)}\| \} \right)^{j/m}, \qquad \forall g \in C^\infty([-1,1]),
\label{Gorny}
\end{equation}
where $\| \, \cdot \, \|$ denotes the sup-norm on $[-1,1]$. Choose $M \geq K$ such that
$$
\sup_{x,t \in \R,|t| \leq 1} \frac{w_K(x+t)}{w_M(x)} < \infty.
$$
Let $p,k \in \Z_+$ and $x \in \R$ be arbitrary. Equation \eqref{ineq-proof-2} yields that for all $m \in \Z_+$ there is $C > 0$ such that
$$
\| \varphi'(x+ \, \cdot \, )\| \leq Cw_M^{1/m}(x) \qquad \mbox{and} \qquad \| \varphi^{(m)}(x+\, \cdot \, )\| \leq Cw_M(x). 
$$
By applying \eqref{Gorny} to $g = \varphi'(x+\, \cdot \, )$ and $m \geq p$ such that
$$
\left(1 - \frac{p-1}{m} \right) \frac{1}{m} + \frac{p-1}{m} \leq \frac{1}{k}
$$
we find that (recall that $w_M \geq 1$)
\begin{align*}
|\varphi^{(p)}(x)| &\leq \| \varphi^{(p)}(x+\, \cdot \,)\|  \\
&\leq C \| \varphi'(x+\, \cdot \, )\|^{1 - (p-1)/m} \left(\max \{ \| \varphi' \| ,  \| \varphi^{(m+1)}\| \}\right)^{(p-1)/m} \\
&\leq C'w_M^{1/k}(x).
\end{align*}
$(iv)\Rightarrow(i)$: As $W$ satisfies \eqref{cond-mult}, this follows from Lemma \ref{lemma-2}.
\\ \\
$(III)$ $(iii) \Rightarrow (ii) \Rightarrow (i)$:  Obvious. \\
 $(i) \Rightarrow (iv)$: Fix arbitrary $p \in \Z_+$ and $N \in \N$. Choose $L \geq N$ such that
$$
\sup_{x,t \in \R, |t| \leq 1} \frac{v_N(x+t)}{v_L(x)} < \infty.
$$
Since $\mathscr{B}_{v_L}  \subset  \mathscr{O}_{M,V}$ and   $ \mathscr{O}_{M,W} \subset \mathscr{O}^p_{M,W}$, we obtain that $C_\varphi(\mathscr{B}_{v_L}) \subset \mathscr{O}^p_{M,W}$. As  $C_\varphi: C^\infty(\R) \to C^p(\R)$ is continuous, De Wilde's  closed graph theorem implies that $C_\varphi: \mathscr{B}_{v_L} \to  \mathscr{O}^p_{M,W}$ is continuous.  Grothendieck's factorization theorem  yields that there is $M \in \N$ such that $C_\varphi: \mathscr{B}_{v_L} \to \mathscr{B}^p_{w_M}$ is well-defined and continuous, and thus that
$$
\| C_\varphi(f)\|_{p, w_M} \leq C\| f\|_{n,v_L}, \qquad \forall f \in \mathscr{B}_{v_L},
$$
for some $n \in \N$ and $C > 0$. Lemma \ref{lemma-1} with $w = w_M$, $v = v_N$ and $\widetilde{v} = v_L$ yields that
 $$
\sup_{x \in \R} \frac{v_N(\varphi(x))}{w_M(x)} < \infty
$$
and (recall that $v_N \geq 1$)
$$
 \sup_{x \in \R} \frac{|\varphi^{(p)}(x)|}{w_M(x)} < \infty.
  $$
$(iv)\Rightarrow(iii)$: As $W$ satisfies \eqref{cond-mult}, this follows from Lemma \ref{lemma-2}.
\end{proof}

\begin{acknowledgement}
L. Neyt gratefully acknowledges support by FWO-Vlaanderen through the postdoctoral grant 12ZG921N.
\end{acknowledgement}

 \end{document}